\documentclass[reqno]{amsart}
\usepackage{amsmath,amsthm,amssymb}
\usepackage{latexsym}
\usepackage{eucal}
\usepackage{srcltx}
\usepackage{eepic,epic,bez123,ebezier,xcolor,curves,calc,graphicx,latexpix,multiply,pgf}

\usepackage{fullpage}

\newtheorem{theorem}{Theorem}[section]
\newtheorem{lemma}{Lemma}[section]
\newtheorem{corollary}{Corollary}[section]

\theoremstyle{definition}

\def\cal{\mathcal}

\begin{document}
\Large

\title[The sharp corner formation in 2d Euler dynamics of patches]
{The sharp corner formation in 2d Euler dynamics of patches: infinite double
exponential rate of  merging}

\author{Sergey A. Denisov}

\email{denissov@math.wisc.edu}


\address{University of Wisconsin-Madison,
Mathematics Department, 480 Lincoln Dr. Madison, WI 53706-1388, USA}

\maketitle

\maketitle
\begin{abstract}
For the 2d Euler dynamics of patches, we investigate the convergence
to the singular stationary solution in the presence of a regular
strain. It is proved that the rate of merging can be double
exponential infinitely in time  and the estimates we obtain are sharp.
\end{abstract} \vspace{1cm}

\section{Introduction and statement of the results}

The two-dimensional Euler equation on the plane can be written in
the vorticity form as follows
\begin{equation}\label{euler-g}
\dot \theta=\nabla \theta \cdot u, \,\, u=\pi \nabla^\perp\Delta^{-1}\theta,\,\, \theta(z,0)=\theta_0,\,z=(x,y)\in
 \mathbb{R}^2, \,\, \nabla^\perp=(-\partial_y,\partial_x)
\end{equation}
The constant $\pi$ in the formula for $u$ can be dropped by time scaling but we write an equation this way on purpose to simplify the later calculations.
The global in time regularity for the smooth initial data dates back
to the paper by Wolibner \cite{Wo}. The method of Wolibner was based
on the Lagrange formulation of the problem but later other
approaches were used (see, e.g., \cite{bm}). Whatever method is used
to establish the global regularity, one proves the upper bounds for
the norms that measure the regularity of the solution and for the 2d
Euler all we know it that these norms grow in time not faster than
the double exponential. The natural question is: are these estimates
sharp? This problem is of course of the same nature as the problem
of the possible blow up for the 3d Euler equation which  is the
central problem in the mathematical theory of fluids. In spite of
its importance, very little is known even in dimension two. For the
two-dimensional torus, it was proved in \cite{den-expexp} that the
vorticity gradient can indeed grow as double exponential for
arbitrarily long (but fixed) time provided that it is large enough
 at time $t=0$. This result, however, (as well as any other known to the author) did
 not reveal an intrinsic mechanism for the singularity formation but rather only
 indicated that the standard double exponential estimates can not be dramatically improved.

In this paper, we describe the scenario in which the singularity
forms though we consider an easier problem, the problem of the patch
evolution. The method we apply is quite general and can be tried for
other evolution equations: 2d and 3d Euler equations, surface
quasi-geostrophic equation, etc.

  The problem of patch evolution deals with the case when the initial data $\theta_0$ in (\ref{euler-g}) is
  the characteristic function of some compact set. The initial data $\theta_0\in L^\infty(\mathbb{R}^2)\cap L^1(\mathbb{R}^2)$ gives rise to globally defined unique solution thanks to the theory of Yudovich \cite{yud}.
 In our case, this compact set will be a centrally
symmetric pair of two simply connected domains with  smooth
boundaries and thus the Yudovich theory assures that the solution
will always be the centrally symmetric pair of domains because the
Euler evolution on the plane preserves the central symmetry. The
main question we address here is: how does the geometry of these
domains change in time?
 This problem  attracted a lot of attention in both
physics and mathematical literature in the last several decades and
became a classical one.  In
\cite{chem}, Chemin proved that if the boundary of the patch is
sufficiently regular then it will retain the same regularity
forever; another proof of that fact was given later by Bertozzi and
Constantin \cite{bc}. We recommend the wonderful books
\cite{bm,chemin2} for introduction to the subject and for simplified
proofs.

Another model which became quite fashionable lately is the so-called surface
quasi-geostrophic equation (SQG for the shorthand). It is different from the Euler equation (\ref{euler-g}) only in the definition of the velocity: $u=\nabla^\perp \Delta^{-1/2}\theta$
 (for SQG) vs.  $u=\nabla^\perp \Delta^{-1}\theta$ (for 2d Euler), i.e., the
 kernel used in convolution is more singular. Only the local in time existence of the
 regular solutions is known for the
 smooth initial data and there is a conjecture that the blow up happens
 in finite time.
The patch evolution for the SQG was also extensively studied  and there is a numerical evidence and widely accepted belief that the patch with smooth boundary can develop the singularity in finite time (see,
e.g., \cite{dc1} and related  \cite{dc2, dc3, cf1, cf2,  mancho}).  The idea behind it is based on the expectation that the stronger the singularity of the kernel is, the faster the process of the singularity formation is supposed to be and since this process should be nonlinear in nature the blow up can supposedly happen in finite time. That mechanism, however, has never been justified or even explained.

 In this paper,
we present the scenario in which the singularity of the kernel plays the major role in the speed of the blow up formation however we have the proof only for the case of 2d Euler.
As we mentioned already, the patch dynamics for 2d Euler can not go wild in finite time so we are looking for the situation when the geometry goes singular at $t=+\infty$.
 The latter can be described in various
ways: growth in time of the curvature, perimeter, etc. In this
paper, we focus on one particular geometric characteristic, the
distance between two interacting patches, and show that it can decay
as double exponential infinitely in time. The estimates we obtain are sharp so in some sence our results are optimal.  That, however, comes with the price as we need to impose the strain whose main purpose is to prevent the bulk of the patches from going into chaotic regime. This will be explained later.
\smallskip

To state the main result we need to introduce some notation first. Let
$\Omega^{-}=\{-z, z\in \Omega\}$, i.e., the image of $\Omega$ under
the central symmetry. The boundary of $\Omega$ will be denoted by
$\Gamma$. In $\mathbb{R}^2\sim \mathbb{C}$, we consider the 2d Euler dynamics
given by the initial configuration $\Omega(0)\cup \Omega^-(0)$ where $\Omega(0)$ will be defined later. The
areas of patches, the distance between them, the value of vorticity
-- all these quantities are of order one as $t=0$ . The curve $\Gamma(0)$ is smooth and its curvature is of order one as well.  As time evolves,
the Euler evolution deforms $\Omega(0)\cup \Omega^{-}(0)$ to a new
pair $\Omega(t)\cup \Omega^-(t)$ with $|\Omega(t)|=|\Omega(0)|$
because the flow preserves the area and the central symmetry. These
new patches will be
 separated from each other for all times, i.e. ${\rm dist}(\Omega(t),
 \Omega^-(t))>0 $, but
the question, however, is how small the distance between them can
get? To answer this question we need to account for the following
well-known fact first.  If one considers the model of Euler
evolution of two identical point vortices on the the plane, then the
dynamics is quite simple: the vortices will rotate with constant
angular velocity \cite{bm}.  That suggests that two symmetric
patches will tend to ``rotate" until some chaotic regime will
homogenize them to the state which is hard to control and there is a
numerical evidence that this chaotic regime does occur for many
$\Omega(0)$ \cite{zabusky}. So, to avoid this chaotic behavior of
our contours (and we need to control them for all $t>0$!), one needs
to impose the strain which will factor out the intrinsic rotation
but which will also be regular enough to not significantly influence
the nonlinear mechanism for the singularity formation. Therefore, it
is very natural to switch from the original problem to the following
one
\begin{equation}\label{euler-g1}
\dot \theta=\nabla \theta \cdot ( u+S), \,\, u=\pi \nabla^\perp
\Delta^{-1}\theta, \quad
\theta(z,0)=\chi_{\Omega(0)}+\chi_{\Omega^-(0)}
\end{equation}
where $\theta(z,t)=\chi_{\Omega(t)}+\chi_{\Omega^-(t)}$ and $S(z,t)$
is sufficiently regular incompressible strain. The analytical
perspective (and this is our way to think about the problem) is that
the solution $\theta$ to (\ref{euler-g1}) we will obtain can be
regarded as  an approximate solution to the original problem
(\ref{euler-g}). Now we are ready to state the main result of the
paper.
\begin{theorem}\label{main}
Let $\delta\in (0,1)$. Then, there is a
simply connected domain $\Omega(0)$ with smooth boundary satisfying
 ${\rm dist}(\Omega(0),\Omega^-(0))\sim
1$, and a time-dependent compactly supported
 incompressible odd strain
\[S(z,t)=(P(z,t),Q(z,t))\] such that
\[
{\rm dist}(\Omega(t),0)\sim {\rm
dist}(\Omega(t),\Omega^-(t))\lesssim \exp(-e^{\delta t}), \quad
t\geq 0
\]
where $\Omega(t)\cup \Omega^-(t)$ is the Euler dynamics of
$\Omega(0)\cup \Omega^-(0)$ in the presence of the strain $S(z,t)$,
i.e. the solution to (\ref{euler-g1}). Moreover, for $S(z,t)$ we
have
\begin{equation}
\sup_{t\geq 0, z\neq 0} \frac{|S(z,t)|}{|z|}<\infty \label{lip}
\end{equation}
and
\begin{equation}\label{log-lip}
|S(z_2,t)-S(z_1,t)|\lesssim |z_1-z_2|(1+|\log|z_1-z_2||)
\end{equation}
uniformly in $z_{1(2)}\in \mathbb{C}$ and $t\geq 0$.
\end{theorem}

{\bf Remark 1.} As it will be clear from the proof, these contours
will touch each other at $t=+\infty$ and the touching point is at
the origin. In the local coordinates around the origin the functions
parameterizing the contours converge to $\pm |x|$ in a self-similar
way which will be described in detail. \smallskip

{\bf Remark 2.} The simple modification of the Yudovich theory (see,
e.g., \cite{bm}) implies that the patch evolution given by
(\ref{euler-g1}) is uniquely defined provided that $S(z,t)$
satisfies (\ref{log-lip}) and is uniformly bounded, the latter is
warranted by (\ref{lip}). Since our strain $S$ is divergence-free
and odd, this dynamics will also preserve the area and the central symmetry. The latter, in particular, implies that the
origin is the stationary point for the dynamics. In the corollary
\ref{cont-l} below, we will give the lower bound on the distance of
any particle trajectory to the origin. This estimate will prove that
the theorem \ref{main} is essentially sharp.\bigskip

Now, we need to explain why the strain we add is indeed a small
correction from the dynamical perspective. This will be done in the
following elementary lemma. We will show that under the $S$--strain
alone no point can approach the origin in the rate faster than
exponential as long as assumption (\ref{lip}) is made and so it is
the nonlinear term $\nabla \theta \cdot u$ in (\ref{euler-g1}) that
produces the ``double exponentially'' fast singularity formation.

\begin{lemma}\label{easy}
Let $S(z,t)$ be an odd  vector field that satisfies (\ref{lip}) and
(\ref{log-lip}). Consider $\theta(z,t)=\chi_{\Upsilon(t)}(z)$ which
solves
\begin{equation}\label{transport}
\dot \theta=\nabla \theta\cdot S(z,t), \quad
\theta(z,0)=\chi_{\Upsilon(0)}(z)
\end{equation}
where ${\rm dist}(\Upsilon(0), 0)>0$ and $\Upsilon(0)$ is a compact
set. Then,
\begin{equation}
{\rm dist}(\Upsilon(t),0)\gtrsim e^{-Ct}\label{lbnew}
\end{equation}
\end{lemma}
\begin{proof}
Clearly (\ref{transport}) is a transport equation and we have
\[
\dot z=-S(z,t),\quad z(0)=z_0
\]
for the characteristics $z(z_0,t)$. It has the unique solution (see
lemma 3.2, page 67, \cite{mp}) due to log-Lipschitz regularity
(\ref{log-lip}). As $S(z,t)$ is odd, $S(0,t)=0$ and so the
origin is a stationary point. The estimate (\ref{lip}) yields
\[
|S(z,t)|<C|z|, \quad \forall t\geq 0\quad  {\rm and}\quad \forall
z\in \mathbb{C}
\]
Therefore,
\[
|\dot r|\leq Cr, \quad r=|z|^2
\]
and
\[
r(0)e^{-Ct}\leq r(t)
\]
This implies (\ref{lbnew}) as the patch $\Upsilon(t)$ is convected
by the flow.
\end{proof}\smallskip

Next, we will show that the theorem \ref{main} is essentially sharp
on the double exponential scale. To explain that, we need the next
lemma. As a motivation let us start with two simple questions: if we
are given a patch $\Omega$ of area $\sim 1$ such that the velocity
generated by $\Omega$ is zero at the origin, how large can the
radial component of the velocity be in the fixed point close to the
origin? How should we choose $\Omega$ to get the maximum (or get
close to the maximum) of this value? Clearly, we are interested in
the radial component since it is the one which will push the points
to the origin or away from it. As the problem is invariant under the
rotation, we can take $z=(x,0), \,\, x>0$ without loss of
generality.
\begin{lemma}\label{shest}
Assume that the set $\Omega$ is such that $|\Omega|\sim 1$ and
\[
u(0)=0, \quad u(z)=\pi \nabla^\perp \Delta^{-1}
\chi_\Omega=(u_1,u_2)
\]
Then,
\[
u_1=x\left(\Theta_\Omega(|z|)+O(1)\right), \quad 0<x<1, \quad
z=(x,0)
\]
and
\[
\Theta_\Omega(|z|)= \int_{\Omega\cap \{|\xi|>|z|\}} \frac{\xi_1\xi_2}{(\xi_1^2+\xi_2^2)^2}d\xi
\]
\end{lemma}
\begin{proof}
This estimate is standard, see lemma~8.1 from
\cite{bm}, pages~315--318. Indeed,
\[
u(z)=u(z)-u(0)=\frac{1}{2}\int_\Omega \left( \frac{z-\xi}{|z-\xi|^2}+\frac{\xi}{|\xi|^2}\right)^\perp d\xi
\]
We have
\[
\int_{\Omega\cap \{|\xi|<2|z|\}} \left| \frac{z-\xi}{|z-\xi|^2}+\frac{\xi}{|\xi|^2}\right| d\xi\lesssim |z|
\]
and
\[
\frac{1}{2}\int_{\Omega\cap \{|\xi|>2|z|\}} \left(
\frac{z-\xi}{|z-\xi|^2}+ \frac{\xi}{|\xi|^2}\right)^\perp
d\xi=(m_1+m_2, \cdot)
\]
where
\[
m_1=-\frac{x^2}{2}\int_{\Omega\cap \{|\xi|>2|z|\}}
\frac{\xi_2}{(\xi_1^2+\xi_2^2)((x-\xi_1)^2+\xi_2^2)}d\xi, \quad\quad
|m_1|\lesssim |z|
\]
and
\[
m_2={x} \int_{\Omega\cap \{|\xi|>2|z|\}}
\frac{\xi_1\xi_2}{(\xi_1^2+\xi_2^2)((x-\xi_1)^2+\xi_2^2)}d\xi
\]
For $m_2$, after scaling by $|z|=x$,
\[
m_2={x} \int\limits_{\widehat \Omega\cap \{|\widehat\xi|>2\}}
\frac{\widehat\xi_1\widehat\xi_{2}}{(\widehat\xi_1^2+\widehat\xi_2^2)
((1-\widehat\xi_1)^2+\widehat\xi_2^2)}d\widehat\xi=x(\Theta_\Omega+O(1))
\]
\end{proof}
\begin{corollary}\label{clara}
For fixed $z=(x,0), x\in (0,1/2)$, we have
\begin{equation}\label{optima}
\max_{\Omega: |\Omega|\sim 1} \Theta_\Omega(|z|)=-\log |z|+O(1)=\Theta_{E\cup E^-}(|z|)+O(1)
\end{equation}
where $E=[0,1]^2$
\end{corollary}
\begin{proof}
The proof is a direct calculation in the polar coordinates.
\end{proof}
In particular, these results give the leading term for the size of
radial component of the velocity field generated by the centrally
symmetric configuration of patches since for this configuration the
velocity is always zero at the origin. We also see that the ``cross
configuration", i.e., $E\cup E^-$, gives the optimal value up to an
additive constant. This configuration will also play a major role in
the proof of theorem \ref{main}.

We can immediately apply two previous results to the problem of
patch dynamics.
\begin{corollary}\label{cont-l}
Under the conditions of the theorem \ref{main}, we have
\begin{equation}
{\rm dist}(\Omega(t),0)\gtrsim \exp(-Ce^{t}) \label{lower-ah}
\end{equation}
where $C$ depends on the constant in (\ref{lip}) and on the initial distance between the contours.
\end{corollary}
\begin{proof}
Indeed, for the trajectory
\[
\dot z=-(u(z,t)+S(z,t)), \quad z(z_0,z)=z_0
\]
we have
\[
\dot r \leq -r (\log r+C), \quad r=|z|^2 \ll 1
\]
which implies (\ref{lower-ah}).
\end{proof}

 The constant $\delta$ from the theorem \ref{main} as
well  as the constant one in front of $t$ in (\ref{lower-ah}) can be
changed by a simple rescaling (i.e., by multiplying the value of
vorticity by a constant which is the same as scaling the time)  thus
the size of $\delta$ is small only when compared to the parameters
of the problem.\smallskip

{\bf Remark 3.} The optimality up to a constant of $E$ in
(\ref{optima}) is the reason why our estimate in the theorem
\ref{main} is essentially sharp on the double exponential scale.
Clearly, one can replace $E$ by any other configuration as long as
it forms a corner of $\pi/2$ at the origin. \smallskip

In $\mathbb{R}^2$, in contrast to $\mathbb{T}^2$,
 the kernel of $\Delta^{-1}$ is easier to write
and $\Delta^{-1}$ can be defined on compactly supported $L^1$
functions so we will address the problem on the whole plane rather
than on $\mathbb{T}^2$. On the 2d torus, similar results hold.
\bigskip

The interaction of two vortices was extensively studied in the
physics literature (see, e.g., \cite{zabusky,saffman}). For example,
the merging mechanism was discussed in  \cite{zabusky} where some
justifications (both numerical and analytical) were given. In our
paper, we provide rigorous analysis of that process and obtain the
sharp bounds.

In \cite{collapse}, the authors study an interesting question of the
``sharp front" formation. Loosely speaking, the sharp front forms
if, for example, two level sets of vorticity, each represented by a
smooth time-dependent curve, converge to a fixed smooth arc as
$t\to\infty$. Let the ``thickness" of the front be denoted by
$\varrho_{front}(t)$. In \cite{collapse}, the following estimate for
2d Euler dynamics is given (see theorem 3, p. 4312)
\[
\varrho_{front}(t)>e^{-(At+B)}
\]
with constants $A$ and $B$ depending only on the geometry of the
front. The scenario considered in our paper is different as the
singularity forms at a point.
\bigskip

The idea of the proof comes from the following very natural
question. Consider an active-scalar dynamics
\begin{equation}
\dot\theta =\nabla\theta\cdot \nabla^\perp (A\theta)
\label{active-sc}
\end{equation}
where $A$ is the convolution  with a kernel $K_A(\xi)$.  In some
interesting  cases, $A=\Delta^{-\alpha}$ with $\alpha>0$ so $K_A$ is
positive and smooth away from the origin. It also obeys some
symmetries inherited from its symbol on the Fourier side, e.g., is
radially symmetric. Perhaps, the most interesting cases are
$\alpha=1$ (2d Euler), which is treated in this paper, and
$\alpha=1/2$ (SQG) already mentioned in the text. If one considers
the problem (\ref{active-sc}) with $A=\Delta^{-\alpha}$ on the 2d
torus $\mathbb{T}^2=[-\pi,\pi]^2$, then there is a stationary
singular weak solution (the author learned about this solution from
\cite{mancho}), a ``cross''
\begin{equation}\label{crosss}
\theta_s=\chi_{E}+\chi_{E^{-}}-\chi_{J}-\chi_{J^{-}}
\end{equation}
where $E=[0,\pi]^2$  and $J=[-\pi,0]\times [0,\pi]$. One can think
about two patches touching each other at the origin and each forming
the right angle. This picture is also centrally symmetric. Now, the
question is: is this configuration stable? In other words, can we
perturb these patches a little so that they will converge to the
stationary solution at least around the origin at $t\to\infty$? The
flow generated by $\theta_s$ is hyperbolic and so is unstable.
However, if one places a curve into the stationary
hyperbolic flow in such a way that its part follows the separatrix of the flow in the attracting direction, then the time evolution of this curve will produce a sharp corner. The
problem of course is that the actual flow is induced by the patch
itself and so it will be changing in time. That suggests that one
has to be very careful with the choice of the initial patch to
guarantee that this process is self-sustaining. Nevertheless, that
seems possible and thus the mechanism of singularity formation
through the hyperbolic flow can probably be justified. We do it here
by neglecting the smaller order terms. In general, the application
of some sort of fixed point argument seems to be needed. Either way,
this scenario is a zero probability event (at least the way the
proof goes) if the ``random'' initial condition is chosen. However,
if one wants to see merging for a long but fixed time, then this can
be achieved for an open set of initial data so from that perspective
our construction is realistic.

We will handle the case $\alpha=1$ only without trying to make the
strain smooth. The question whether one can choose $S(z,t)\in
C^\infty$ seems to be the right one to address (rather than trying
to make $S(z,t)=0$) and we formulate it as an\smallskip

{\bf Open problem.} {\it Can one take the strain $S(z,t)$ in the
theorem \ref{main} to be infinitely smooth?}\smallskip

If the answer is yes, then one can view this strain as the
incompressible flow generated by the some background or another
patch far away which is decoupled from the evolution of our
symmetric pair (that of course is yet another approximation among
many others made in the area of Fluid Dynamics). It is not easy for
us to imagine how to get rid of the strain completely for the 2d
Euler model if one wants to prove the sharp estimate on the rate of
merging. For other models like SQG, this is not ruled out as one has
to control the bulk of the patches only for finite time and for this
finite time the chaotic regime might not be able to start
influencing the picture.

Solving the open problem stated above might require closure
of the fixed point argument outlined above. We were able to do that
so far only for the model equation where the convolution kernel is a
smooth bump. Surprisingly, the analysis needed to justify the
self-sustaining process we discussed is reminiscent to what one has
to do to prove the homogenization in the 2d Euler for $t\to\infty$
(see, e.g., \cite{cag-maf}). This is technically hard and has never
been carried out. There is yet another argument indicating that the
answer to the open problem above is positive. In \cite{saffman}, the
existence of very interesting $V$-shapes is mentioned (no proofs
though). The $V$-shape is a patch that rotates with the constant
angular velocity under the 2d Euler dynamics. Saffman indicates that
there is a continuous parametric curve of these $V$-shapes, each
represented by a pair of centrally symmetric patches with smooth
boundary. The endpoint of this curve however is represented by a
pair of centrally symmetric patches that touch each other at the
origin and form a sharp corner there (a singularity we want to produce dynamically). It would be interesting to
prove existence of this curve and show that the dynamics
representing the evolution of our patch can move across these
``invariant sets" (i.e., the particular $V$--shapes) thus
approaching the endpoint of this parametric curve locally around the
origin. That would be consistent with the self-sustaining scenario
we want to understand.
\smallskip

For $\alpha>1$, one can show that the merging happens but the
attraction to the origin is only exponential. This case is much
easier as the convolution kernel is not singular anymore so, for
example, one can take the strain to be exponentially decaying in time.

For $\alpha<1$, we expect our technique to show that the contours
can touch each other in finite time thus proving the outstanding
problem of blow up for $\alpha=1/2$. This, however, will require
serious refinement of the method.\smallskip

There are other stationary singular weak solutions known for 2d
Euler dynamics and for other problems in fluid mechanics. It would
be interesting to perform analogous stability analysis for each of
them with the goal of, e.g., solving the problems of blow up. From
that perspective the idea is quite general: find the singular
stationary solution which generates the ``hyperbolic dynamics" and
construct the stable manifold around it which will belong to the
functional space of high regularity. For the ``true" 2d Euler or SQG
one can try some modification of the same ``cross
configuration".\smallskip

The structure of the paper is as follows. In the sections 2, 3, and
4, we prove some auxiliary results. The section 5 contains the
construction of  $\Omega(t), S(z,t)$ and the proof of the theorem
\ref{main}.  In appendix, we find an approximate self-similar
solution to the local equation of  curve's evolution.

\section{Velocity field generated by the limiting configuration}

In this section,  we consider the velocity field generated by one
particular pair of patches which resembles around the origin the
limiting ($t=\infty$) configuration of the dynamics described in
theorem \ref{main}. In $\mathbb{R}^2$, take
$\theta_s(z)=\chi_{E}(z)+\chi_{E^-}(z)$, where $E=[0,1]^2$. Recall
that in the case of $\mathbb{T}^2$ analogous configuration
(\ref{crosss}) is a steady state. Let $x,y>0$, the other cases can
be treated using the symmetry of the problem. We have
\[
 \pi \nabla \Delta^{-1}\theta_s= \frac 12\int_{E\cup E^-}
\frac{(x-\xi_1,y-\xi_2)}{(x-\xi_1)^2+(y-\xi_2)^2}d\xi_1 d\xi_2
\]
Integrating, we have for the first component
\[
\frac 14 \int_0^1 \Bigl(\log (x^2+(y+\xi_2)^2)-\log(x^2+(y-\xi_2)^2)\Bigr)
d\xi_2+r_1(x,y)
\]
\[
=\frac 12 \int_0^y \log (x^2+\xi^2)d\xi+r_2(x,y)
\]
where $r_{1(2)}(x,y)$ are odd and smooth around the origin.
Integrating by parts, we get the following expression for the
integral
\[
y\log (x^2+y^2)-2y+2x\arctan ({y}/{x})
\]
By symmetry, for the second component of the gradient we have
\[
x\log(x^2+y^2)-2x+2y\arctan(x/y)
\]
Thus, the velocity is
\begin{equation}\label{ididi}
u_s(z)=\pi \nabla^\perp \Delta^{-1}\theta_s=\frac 12 (-x,y)\log(x^2+y^2)+r(x,y)
\end{equation}
around the origin, where $|r(z)|\lesssim |z|$. The correction
$r(x,y)$ has a bounded gradient in $0<x,y<1$ and so it belongs to
the Lipschitz class. This ensures that the first term in
(\ref{ididi}) is log-Lipschitz.

Consider a positive function $\Phi(\phi)$ on the unit circle
$\phi\in \mathbb{T}$, $\pi/2$-periodic, smooth,
$\Phi(\phi)=\Phi(\pi/2-\phi)$, and such that
\[
\Phi(\phi)=\cos \phi, |\phi|<\phi_0=\arctan 0.5;\quad
\Phi(\phi)=\sin\phi, \pi/2-\phi_0<\phi<\pi/2+\phi_0
\]
Define the following potential:
\begin{equation}\label{l0}
\Lambda_s(z)=xy\log Q(x,y),\quad Q(x,y)=|z|\Phi(\phi), \quad {\rm
where}\,\, z=|z|e^{i\phi}
\end{equation}
Clearly, $Q$ is homogeneous of order one and
\begin{equation}\label{horosho}
Q(x,y)=x\,\,{\rm if}\,  |\phi|<\phi_0;\quad Q(x,y)=y\,\,{\rm if}\,
|\phi-\pi/2|<\phi_0
\end{equation}

\begin{lemma}
Around the origin, we have
\begin{equation}
u_s(z)=\nabla^{\perp} \Lambda_s(z)+r(z) \label{rs1}
\end{equation}
where $r(z)$ is log-Lipschitz and
\begin{equation}\label{sbg}
 |r(z)|\lesssim |z|
\end{equation}
\end{lemma}
\begin{proof}
It is sufficient to handle $x,y>0$. One can write
\[
\nabla^{\perp} \Lambda_s(z)=\frac 12 (-x,y)\log(x^2+y^2)+r_3(z)
\]
Then, it is a direct calculation to check that $|r_3(x,y)|\lesssim
|z|$ and that $\nabla r_3(z)$ is bounded.
\end{proof}
{\bf Remark 1.} This particular choice of $\Lambda_s$ was made only
to simplify the calculations below.

We will focus later on the first term in (\ref{rs1}) as the second
one has a smaller size and will be absorbed into the strain later
(see the formulation of the main theorem). The level sets of
$\Lambda_s$ around the origin are hyperbolas asymptotically. Indeed,
take $x=\epsilon \widehat x, y=\epsilon \widehat y$ and consider
\[
\Lambda_s(z)=\epsilon^2 \log \epsilon
\]
or
\begin{equation}\label{hyperbola}
 \widehat x\widehat
y\left(1+\frac{\log|\widehat z|}{\log\epsilon}+\frac{\log
\Phi(\phi)}{\log \epsilon}\right)=1
\end{equation}
and so the second and the third terms are small as $\epsilon\to 0$
provided that $|\widehat z|\sim 1$.

\smallskip For the velocity,
\begin{equation}\label{ode}
\nabla^\perp \Lambda_s(z)= \left\{
\begin{array}{cc}
(-x(\log y+1), y\log y), & \phi\in (\pi/2-\phi_0,\pi/2)\\
(-x\log x, y(\log x+1)), & \phi\in (0,\phi_0)
\end{array}
\right.
\end{equation}
Within the sector $\phi\in (\phi_0,\pi/2-\phi_0)$, the formula for
the gradient is more complicated but we will see later that the part
of the patch's boundary that belongs to this sector will not have a
significant contribution to the velocity and so very rough estimates
will suffice.

These formulas show that the flow generated by $\Lambda_s$ is
hyperbolic around the origin. Moreover, the attraction and repelling
is double exponential. For example, the point $(0,y_0)$ will have a
trajectory $(0,y_0^{e^t})$. In the next section, we will study the
Cauchy problem associated to this flow.

{\bf Remark 2.} Later on we will also need to handle the vector
field generated by the ``smoothed cross configuration". Suppose we
have a small parameter $\epsilon$ and
\begin{equation}\label{sglad}
E(\epsilon)=E\backslash \{|z|<\epsilon\}
\end{equation}
Then,  for $\nabla^\perp \Delta^{-1} \chi_{E(\epsilon)}$, the direct
calculation analogous to the one done in lemma \ref{shest} yields
\begin{equation}\label{sem}
\pi \nabla^\perp \Delta^{-1} \chi_{E(\epsilon)\cup
E(\epsilon)}=\nabla^\perp \Bigl( xy\log \epsilon \Bigr)+{O}(|z|),
 \quad |z|\lesssim \epsilon
\end{equation}\bigskip

\section{The Cauchy problem for the dynamics generated by
the limiting configuration}\label{cauchy}

The construction of the patches and the strain in the main theorem
will be based on the calculations done in this section.  We consider
the following Cauchy problem where $t$ is a parameter:
\[
\dot z(\tau,t)=\nabla^\perp \Lambda_s(z(\tau,t)), \quad\tau\geq 0
\]
the formula for the right hand side is given in (\ref{ode}), and the
initial position is
\begin{equation}
z(0,t)=(\epsilon(t), e^{-1}), \quad \epsilon(t)=\exp(-e^{\varrho
t})\label{koshi}
\end{equation}
For now, $\varrho$ is some fixed positive number.   So, at each time
$t$, a point (let us say it has an {\bf index} $t$) with  initial
position $(\epsilon(t),e^{-1})$ starts moving under the flow so that
at time $t+\tau$ (i.e., time $\tau$ spent since the start of the
motion) we will see the arc built by points that are indexed by
$t+\tau_1, \tau_1\in [0,\tau]$. Let us call this arc $\Gamma_1(t)$
(rotated by $\pi/4$ in the anticlockwise direction, it will be a
part of the $\Gamma(t)=\partial\Omega(t)$ from the main theorem).

Our goal in this section is to study the $\Gamma_1(t)$ around the
origin when $t\to+\infty$. This is a straightforward and rather
tedious calculation which we have to perform.

For each $t$, consider the trajectory $z(\tau,t)$ within the fixed
neighborhood of the origin: $\{[0,e^{-1}]\times [0,e^{-1}]\}$. This
trajectory will be an invariant set for potential $\Lambda_s$ that
corresponds to $\Lambda_s=-e^{-1}\epsilon(t)$. The parameter $t$,
again, can be regarded as the tag of this trajectory. For
$t\to\infty$, these trajectories will look more and more like
hyperbolas around zero due to (\ref{hyperbola}).

There will be three important events for each trajectory: the first
one is when it crosses the ray $\phi=\pi/2-\phi_0$ (i.e., $y=2x$) at
time $\tau=T_1(t)$, the second one is when it crosses the ray
$\phi=\phi_0$ (i.e., $y=x/2$) at time $T_2(t)$, and the third one is
when it crosses the line $x=e^{-1}$ and thus leaves the domain of
interest $0<x<e^{-1}, 0<y<e^{-1}$. We will denote this time by
$T_3(t)$. Since the trajectory is an invariant set for $\Lambda_s$,
\begin{equation}\label{fact-tr}
z(t,T_1(t))=(\widehat\epsilon(t), 2\widehat\epsilon(t))
\end{equation}
and $\widehat\epsilon (t)$ can be found from the equation (see
 (\ref{horosho}))
\begin{equation}
2\widehat \epsilon^2(t) \log (2 \widehat
\epsilon(t))=-e^{-1}\epsilon(t)\label{ura1}
\end{equation}
This gives
\begin{equation}
\widehat\epsilon(t)= \sqrt{\frac{\epsilon(t)}{e\log \epsilon^{-1}(t)
}}\left(1+\bar{o}(1)\right), \quad t\to\infty \label{asr}
\end{equation}
One can write an asymptotical expansion up to any order but we are
not going to need it.

Within $\phi>\pi/2-\phi_0$, we have $\dot y=y\log y, \,
y(0,t)=e^{-1}$ and so
\[
y(\tau,t)=e^{-e^\tau}, \quad \tau<T_1(t)
\]
Then, the value of $T_1(t)$ can be found from
\[
e^{-e^{T_1(t)}}=2\widehat{\epsilon}(t)
\]
or (due to (\ref{ura1}))
\[
2e^{T_1}-T_1=1-\log 2+e^{\varrho t}
\]
That gives us an  asymptotics
\begin{equation}
T_1(t)=\varrho t-\log 2+e^{-\varrho t}(1-2\log 2+\varrho
t)+O(t^2)e^{-2\varrho t}  \label{asii1}
\end{equation}
Therefore, for the actual time $w_1=t+T_1(t)$, we have
\[
w_1=t(1+\varrho)-\log 2+e^{-\varrho t}(1-2\log 2+\varrho
t)+O(t^2)e^{-2\varrho t}
\]
the curve we study will intersect the line $y=2x$ at the point
$(\widehat\epsilon(t),2\widehat\epsilon(t))$. Then, assuming that
$t$ is a function of $w_1$, we have
\[
t(w_1)=\frac{w_1}{1+\varrho} + \frac{\log 2}{1+\varrho} -
\frac{2^{-\delta}(\delta w_1+1-2\log 2+\delta \log
2)}{1+\varrho}e^{-\delta w_1}+O(e^{-(1+\epsilon_3)\delta w_1})
\]
\begin{equation}\label{ohoh}
\delta=\varrho/(1+\varrho)<1,\quad  \epsilon_3>0
\end{equation}
and
\begin{equation}\label{omg}
f_1(w_1)=\widehat\epsilon(t(w_1))=F_1(\varrho)\exp\left(-e^{\delta
w_1}
2^{-1/(\varrho+1)}+\frac{\delta(\delta-1)}{2}w_1\right)(1+\overline{o}(1))
\end{equation}
where $F_1(\varrho)$ can be computed explicitly.


\subsection{The asymptotical form of the curve within
$e^{-1}>y>2x>0$.}

 We first compute the scaling limit of the curve in the
$\left\{|z|<\Bigl(f_1(w_1)\Bigr)^\kappa\right\}$ neighborhood of zero
where $\kappa$ is any small positive number. Fix the time $w_1$ and
solve the system backward in time using auxiliary functions
$\alpha(\tau, w_1), \beta(\tau,w_1)$:
\begin{equation}\label{ode5}
\dot\alpha=\alpha(\log \beta+1),\quad \dot\beta=-\beta\log \beta
\end{equation}
and $\alpha(0)=f_1(w_1), \beta(0)=2f_1(w_1)$. We get then
\begin{equation}\label{sravni}
\beta(\tau,w_1)=\left(2f_1(w_1)\right)^{e^{-\tau}}, \quad
\alpha(\tau,t)=e^\tau \left(f_1(w_1)\right)^{2-e^{-\tau}}
\end{equation}
Take $T$, the moment at which the curve is studied, and let
\[
T=w_1-\tau, \quad \alpha=f_1(T) \widehat \alpha, \quad \beta=2f_1(T)
\widehat \beta
\]
Given $\widehat \alpha$ and $\widehat \beta$, the parameters $w_1$
and $\tau$ are now the functions of $T$, $\widehat \alpha $, and
$\widehat \beta$. We rescale the variables $(x,y)$ as
\[
x=f_1(T)\widehat x, \quad y=f_1(T)\widehat y
\]
The curve we study will go through the point $\widehat x=1, \widehat
y =2$ after this rescaling and this is the normalization we need.
Will the rescaled curve have any limiting behavior? To answer this
question, we consider equations
\[
e^\tau(f_1(w_1))^{2-e^{-\tau}}=f_1(w_1-\tau)\widehat \alpha , \quad
(2f_1(w_1))^{e^{-\tau}}=2f_1(w_1-\tau)\widehat\beta
\]
Denoting $u=\log\widehat\alpha, v=\log\widehat\beta$, we get
\begin{equation}\label{ratioo1}
\displaystyle \frac{v}{u}=\frac{(e^{-\tau}-1)\log 2+e^{-\tau}\log
f_1(w_1)-\log f_1(w_1-\tau)}{(2-e^{-\tau})\log f_1(w_1)+\tau-\log
f_1(w_1-\tau)}
\end{equation}
Now, let $\tau=e^{-\delta w_1}\widehat \xi$ with $\widehat
\xi=\overline{o}(e^{\delta w_1})$. From (\ref{omg}), we immediately
get
\begin{equation}\label{ratioo}
\frac{v}{u}\to -\omega,\quad \omega=\frac{1-\delta}{1+\delta}, \quad
{\rm as }\quad  T\to\infty
\end{equation}
and thus the rescaled curve will converge uniformly to the graph of
the function $\widehat y=2\widehat x^{-\omega}$ on any interval
$\widehat x\in [\widehat a,1]$ with fixed $\widehat a$ (this is the
regime of fixed $\xi$). If $\widehat \xi=\overline{o}(e^{\delta
w_1})$, then the corresponding interval is $\widehat x\in
[f_1^{\varpi(w_1)}(w_1),1]$ (here $\varpi(w_1)$ is a positive function
converging to zero arbitrarily slowly) and on that interval we have
(\ref{ratioo}) uniformly. This implies $\widehat y=2\widehat
x^{-\omega+\overline{o}(1)}$ on that interval for the rescaled curve
we study and the convergence $\overline{o}(1)\to 0$ is uniform.

If one fixes $\tau>0$ in (\ref{ratioo1}) instead (the regime of
$\widehat\xi\sim e^{\delta w_1}$), then
\[
\frac{v}{u}\to \frac{e^{-\tau}-e^{-\delta
\tau}}{2-e^{-\tau}-e^{-\delta \tau}}=-s_-(\tau,\delta), \quad{\rm
as}\, T \to \infty
\]
First, notice that
\[
s_-(\tau,\delta)>0
\]
and $ s_-(\tau,\delta)\to \omega, \tau\to 0$. Now, we see that,
depending on $\tau\geq 0$, we have different zones for $x$ and
different asymptotical regimes
\[
\widehat y=2\widehat x^{-s_-(\tau,\delta)+\overline{o}(1)}
\]
For example, if $\tau\to 0$, the limiting behavior  is again
\begin{equation}
\widehat y=2\widehat x^{-\omega+\overline{o}(1)} \label{sh1}
\end{equation}
Notice now that $\delta\to 0$ as $\varrho\to 0$ and $ \omega\to 1
\quad{\rm if } \quad\delta\to 0 $ so in the $\varrho\to 0$ limit the
limiting shape for $\tau\ll 1$ is hyperbola, which is approximately
the invariant set for $\Lambda_s$. This makes the perfect sense as
the initial point $(e^{-e^{\varrho t}}, e^{-1})$ approaches the
separatrix $OY$ slower and slower when $\varrho\to 0$. Had it been
constant in $t$, the curve would exactly follow the invariant set
for $\Lambda_s$.\smallskip

{\bf Remark 1.} For the distance to the origin we have
\[
\min_{\widehat x>0}
\Bigl(\widehat{x}^2+4\widehat{x}^{-2\omega}\Bigr)= \widehat{x}_m^2+
4\widehat{x}_m^{-2\omega}<\sqrt{1+4}=\sqrt{5}
\]
where
\[
x_m=(4\omega)^{1/(2\omega+2)}
\]

For any fixed $\tau>0$, we computed the asymptotical shape of the
curve but we will also need the bounds on the curve for $\tau(T)\to
+\infty$. In the next section we will need to control the ratio $x/y$ for the points on the curve.
From (\ref{sravni}), we get
\begin{equation}\label{ugolok}
\frac{x}{y}=e^\tau 2^{-e^{-\tau}}(f_1(w_1))^{2(1-e^{-\tau})}< (f_1(T))^{2(1-e^{-\tau_0})}
\end{equation}
for $\tau>2\tau_0>0$. Indeed,  $w_1=T+\tau$ and so $e^\tau (f_1(w_1))^{\epsilon_3}<e^\tau  (f_1(\tau))^{\epsilon_3}\lesssim 1$ for any  $\epsilon_3>0$.

\subsection{The behavior of the curve in the sector $x/2<y<2x$}

 Similarly to (\ref{fact-tr}), we have
\[
z(t,T_2(t))=(2\widehat \epsilon(t),\widehat \epsilon(t))
\]
since the level sets for $\Lambda_s$ are symmetric with respect to
the line $y=x$.

 Let us compute the asymptotics of $T_2(t)-T_1(t)$, i.e.
the time it takes for the trajectory with index $t$ to pass through
the sector. Notice that inside this sector $x\sim y$ and so  we have
the following system of equations
\[
\left\{
\begin{array}{cc}
\dot\alpha=-\alpha(\log \alpha+H_1(\alpha,\beta)), &\alpha(0)=\widehat{\epsilon}(t)\\
\dot\beta=\beta(\log\beta+H_2(\alpha,\beta)) ,& \beta(0)=2\widehat{\epsilon}(t)
\end{array}
\right.
\]
where $H_{1(2)}\sim 1$.
We can find $T_2-T_1$ then from
\[
T_2-T_1=\int_0^{T_2-T_1} \frac{\dot \alpha d\tau}{\alpha (\log \alpha^{-1}-H_1(\alpha,\beta))}
\]
we have
\begin{equation}\label{ocenka}
\log\left(  \frac{\log\widehat\epsilon^{-1}(t)+C}{\log\widehat\epsilon^{-1}(t)+C-\log 2}   \right)<T_2-T_1<\log\left(  \frac{\log\widehat\epsilon^{-1}(t)-C}{\log\widehat\epsilon^{-1}(t)-C-\log 2}   \right)
\end{equation}
and
\[
T_2-T_1=\frac{\log 2}{\log\widehat\epsilon^{-1}(t)}+O(\log^{-2}\widehat\epsilon (t))=(2\log 2) e^{-\varrho t}+O(t e^{-2\varrho t})
\]
So, at the actual time $w_2=t+T_2(t)$,
\[
w_2=t(1+\varrho)-\log 2+e^{-\varrho t}(1+\varrho
t)+O(t^2e^{-2\varrho t})
\]
the trajectory intersects the ray $y=x/2$
at the point $(2\widehat \epsilon(t), \widehat\epsilon(t))$. If one writes
$t$ as a function in $w_2$, then
\[
(2\widehat \epsilon(t), \widehat\epsilon(t))=(2f_2(w_2), f_2(w_2))
\]
with
\[
f_2(w_2)=F_2(\varrho)\exp\left(-e^{\delta w_2}
2^{-1/(\varrho+1)}+\frac{\delta(\delta-1)}{2}w_2\right)(1+\overline{o}(1))
\]
(compare it to (\ref{omg})). The simple calculation yields that
\begin{equation}\label{ravno}
f_2(T)\sim f_1(T)
\end{equation}
One can be more precise here: repeating the calculations
(\ref{sravni}) and doing the estimation similar to (\ref{ocenka}),
one can prove that the curve scaled to go through the point
$\widehat x=1, \widehat y=2$ (i.e., we scale by $f_1(T)$ at any time
$T$) converges uniformly to the graph of the function $\widehat
y=2\widehat x^{-\omega}$. This is achieved again by scaling the
local time $\tau$ as $\tau=e^{-\delta w_1}\widehat \xi$.

\subsection{ The case $0<y<x/2<e^{-1}/2$}

This part of the curve is more complicated but the analysis is
nearly identical to what we did in the first subsection. Consider
equations
\[
\dot\alpha(\tau)=-\alpha(\tau) \log \alpha(\tau), \quad
\dot\beta(\tau)=\beta(\tau)(\log \alpha(\tau)+1)
\]
with initial conditions $\alpha(0)=2f_2(w_2), \beta(0)=f_2(w_2)$. We
are interested in the shape of the curve at time $T=w_2(T)+\tau$
where $w_2(T)$ is the actual time in the past when the trajectory
intersected the ray $y=x/2$.
\begin{equation}\label{abu}
\alpha(\tau)=(2f_2(w_2))^{e^{-\tau}}, \quad
\beta(t)=(f_2(w_2))^{2-e^{-\tau}}e^{\tau}
\end{equation}
We again scale by $f_2(T)$ as follows
\[
\alpha=2f_2(T)\widehat\alpha, \quad\beta=f_2(T)\widehat\beta
\]
and then rescale the variables as $x=f_2(T)\widehat x$ and
$y=f_2(T)\widehat y$. If $u=\log \widehat\alpha, v=\log
\widehat\beta$ then
\[
\frac{v}{u}=\frac{(2-e^{-\tau})\log f_2(w_2)+\tau-\log
f_2(w_2+\tau)}{ (e^{-\tau}-1)\log 2+e^{-\tau}\log f_2(w_2)-\log
f_2(w_2+\tau)}
\]
and, again, if $\tau=\widehat \xi e^{-\delta w_2}$, then
\begin{equation}
\frac{v}{u}\to -\omega
\end{equation}
and thus the curve, rescaled to go through the point $\widehat x=2,
\widehat y=1$, will converge to the graph of the function $\widehat
y=(\widehat x/2)^{-\omega}$ uniformly on any interval $[2,\widehat
b]$ where $\widehat b$ is fixed. For the general case of $\widehat
\xi=\overline{o}(e^{\delta w_2})$ we have $\widehat y=(\widehat
x/2)^{-\omega+\overline{o}(1)}$ for the rescaled curve uniformly on
the interval $\widehat x\in [2, f_2^{-\varpi(w_2)}(w_2)]$ at time
$T=w_2+\tau$.

For $\tau>0$ fixed, we have
\begin{equation}
\frac{v}{u}\to \frac{2-e^{-\tau}-e^{\delta
\tau}}{e^{-\tau}-e^{\delta \tau}}=-s_+(\tau,\delta)
\label{limsh}
\end{equation}
and we again have
\[
s_+(\tau,\delta)\to \omega, \quad {\rm as}\quad \tau\to 0
\]

For any fixed $\tau\in (0, \tau_{cr})$, we have
\[
s_+(\tau,\delta)=-\frac{2-e^{-\tau}-e^{\delta
\tau}}{e^{-\tau}-e^{\delta \tau}}>0
\]
The critical $\tau_{cr}(\delta)=\log \xi_{cr}(\delta)$ where
$\xi_{cr}(\delta)$ is the solution of the equation
\[
2-\frac{1}{\xi}-\xi^{\delta}=0,\quad  \xi>1
\]
Notice that $\xi_{cr}(\delta)\to 1$ as $\delta\to 1$ and
$\xi_{cr}(\delta)\to +\infty$ as $\delta\to 0$. Consequently,
for the critical value
\[
\tau_{cr}(\delta)\to 0, \quad \delta\to 1
\]
and
\[
\tau_{cr}(\delta)\to +\infty, \quad \delta\to 0
\]
As before, for the subcritical value $\tau<\tau_{cr}$, the curve
will have a limiting shape given by
\[
\widehat y=(\widehat x/2)^{-s_+(\tau,\delta)+\overline{o}(1)}
\]
however $s_+(\tau,\delta)$ decays in $\tau$ and
$s_+(\tau_{cr},\delta)=0$.

The $(x,y)$--coordinates of part of the curve corresponding to each
$\tau<\tau_{cr}$ is easy to find. We again have
\[
f_2(T)=F_2(\varrho)\exp\left(-e^{\delta T}
2^{-1/(\varrho+1)}+\frac{\delta(\delta-1)}{2}T\right)(1+\overline{o}(1))
\]
and
\begin{equation}\label{primerno}
x\sim (f_2(T))^{d_2(\tau)+\overline{o}(1)},\quad
d_2(\tau)=e^{-\tau(1+\delta)}<1
\end{equation}
\[
y\sim (f_2(T))^{l_2(\tau)+\overline{o}(1)}, \quad
l_2(\tau)=2e^{-\delta\tau}-e^{-\tau(1+\delta)}
\]
and $l_2(\tau)>1$ for $\tau<\tau_{cr}$. In the $x$--coordinate, the
domain corresponding to $\tau<\tau_{cr}$ will be characterized by
$2f_2(T)<x<f_2^{d_2(\tau_{cr}-\epsilon_4)}(T)$ where
$\epsilon_4<\tau_{cr}-\tau$. Moreover,
\[
d_2(\tau_{cr})\to 0, \quad {\rm as}\quad \delta\to 0
\]
\smallskip
What can be said about the curve for the region with $\tau\geq
\tau_{cr}-\epsilon_4$ where $\epsilon_4$ is small? For the analysis
that follows in the next section we will only need very rough
bounds.

Starting at time $w_2$, the point with coordinates $(2f_2(w_2),f_2(w_2))$ will
go along the trajectory which will then cross the line $x=e^{-1}$ in
time
\[
\tau_{full}=\log\log \frac {1}{2f_2(w_2)}\sim \delta w_2-\frac{\log
2}{1+\varrho}+\bar{o}(1)
\]
and so $ T_3(t)=w_2+\tau_{full}=(1+\delta)(1+\rho)t+
(1+\delta-(\rho+1)^{-1})\log 2+\overline{o}(1) $. Therefore, for the
part of the curve that corresponds to $\tau>\tau_{cr}-\epsilon_4$ at
time $T$, we have
\begin{equation}\label{predel1}
x>(f_2(T))^{d_2(\tau_{cr}-2\epsilon_4)}
\end{equation}
from (\ref{primerno})  and $d_2(\tau_{cr}-2\epsilon_4)<1$. In the
next section, we will be interested in the ratio $y/x$ for every
point on this part of the curve. We use (\ref{abu}) to get
\[
\frac{y}{x}=e^\tau 2^{-e^{-\tau}}(f_2(w_2(T))^{2-2e^{-\tau}}<e^\tau (f_2(w_2(T)))^{\zeta}, \, \zeta=2-2e^{-(\tau_{cr}-\epsilon_4)}, \quad \epsilon_4>0
\]
and
\[
\zeta>0
\]
However, the formulas for $\tau_{full}$ and
$T_3$ indicate that $T<(1+\delta)w_2(T)$. Therefore, part of the
curve corresponding to $\tau>\tau_{cr}-\epsilon_4$ will have
\begin{equation}\label{predel2}
\frac{y}{x}<e^{\tau}(f_2(T/(1+\delta)))^\zeta<(f_2(T/(1+2\delta)))^\zeta
\end{equation}
for large $T$ since $f_2(T)$ decays as double exponential.

\subsection{Self-similar behavior around the origin}

At any given time $T$, denote the point where the curve intersects
the line $y=x$ by $(f(T),f(T))$. The calculations given above show
that
\begin{equation}\label{omg1}
f(T)\sim f_1(T)\sim f_2(T)
\end{equation}
 We also saw that the curve
rescaled by this $f(T)$ (and thus normalized to go through the point
$\widehat x=1, \widehat y=1$) will converge uniformly to the graph
of the function $\widehat y=\widehat x^{-\omega}$. How large is the
domain of convergence? If the curve is calculated at time $T$ and if
$\varpi(T)$ is arbitrary positive function such that $\varpi(T)\to 0$ as
$T\to\infty$, we have
\begin{equation}\label{asima}
\widehat y=\widehat x^{-\omega+\overline{o}(1)}
\end{equation}
for the curve at time $T$ uniformly on
\[
\Bigl(f(T)\Bigr)^{\varpi(T)}<\widehat x<\Bigl(f(T)\Bigr)^{-\varpi(T)}
\]
where $f(T)\sim f_1(T)\sim f_2(T)$ is the rescaling parameter.
Outside this window, we have different scaling limits for
 different values of the parameter $\tau>0$.\smallskip

{\bf Remark 2.} What is the nature of the scaling law for $\tau\sim
0$ that we have got? The Euler dynamics is defined by the
convolution with the kernel
\[
K(z)\sim \log |z|
\]
and $|\log (s\epsilon)|=|\log\epsilon|+O(1)$ for $s\in [1,M]$ with
any fixed $M$. Thus, the strength of the created hyperbolic flow  is
more or less the same within any annulus $\epsilon<|z|<M\epsilon$.
Outside this annulus, say, for $|z|=\sqrt\epsilon$, the size of the
kernel is quite different.

Now, assume that we are given the hyperbolic flow in the whole plane
defined by the following equations
\[
\left\{
\begin{array}{c}
\dot x=y \\
\dot y=x
\end{array}
\right.
\]
One obtains these equations after rotating the phase space in, e.g.,
(\ref{ode5}) by $\pi/4$ degrees.

Let us find the evolution of the contour $\cal{C}(t)$ under this
flow which would be self-similar in the sense that
$\cal{C}(t)=\cal{C}(0)e^{-\delta t}$. Assume that $\cal{C}(t)$ is
given by the graph of the function $y(x,t)$ and then one gets
nonhomogeneous Burgers equation for $y(x,t)$
\[
y_t(x,t)=-y_x(x,t)y(x,t)+x
\]
By our assumptions,
\[
y(x,t)=e^{-\delta t} H(e^{\delta t}x)
\]
so we have
\begin{equation}
\delta(\xi H'(\xi)-H(\xi))=\xi-H'(\xi)H(\xi) \label{oldstory}
\end{equation}

\[ H'=\frac{\delta H+\xi}{\delta \xi+H}
\]
If $H=\xi z(\xi)$, then
\[
|z-1|^A|z+1|^B=C|\xi |^{-1}, A=\frac{1+\delta}{2},
B=\frac{1-\delta}{2}
\]
and so
\[
|H-\xi|^A|H+\xi|^B=C
\]
Due to scaling, we can assume $C=1$. In the coordinates
$H-\xi=\beta, H+\xi=\alpha$, we have
\[
\beta=\alpha^{-(1-\delta)/(1+\delta)}
\]
(compare with (\ref{asima})).

Is it possible to find the initial data such that the evolution of
the curve is self-similar on the larger interval? The answer to this
question is yes. For example, in (\ref{koshi}), one can take
\[
\epsilon(t)=\exp\Bigl(-e^{t^{\widetilde\gamma}}\Bigr)
\]
with $\widetilde\gamma\in (0,1)$. Then, the distance from the curve
to the origin will be $\lesssim
\exp\Bigl(-e^{t^{\widetilde\gamma_1}}\Bigr),$ $\widetilde\gamma_1\in
(0,1)$ but the self-similar behavior will take place in a wider
relative range (but yet not on the whole ball $B_{e^{-1}}(0)$!).
This is a general rule: the slower the Cauchy data jumps from one
invariant set of $\Lambda_s$ to another, the more regular the curve
is around the origin.\bigskip

\section{Comparison of the velocity fields generated by $\Gamma_1(t)$ and by the
limiting configuration. }

In this section, let the symbol $t$ denote the actual time at which
the curve is considered and let $\epsilon=f(t)$ where $f(t)$ was
introduced earlier. Recall that we denoted by $\Gamma_1(t)$ the arc
constructed in the previous section. Let us take $\Gamma_1(t)$ and
close it in the smooth and arbitrary way to produce the
simply-connected domain $\Omega_1(t)$ (see Figure 1). Then,  ${\rm
dist} \Bigl( \Omega_1(t), 0\Bigr)\sim f(t)$. We will compare now the
velocity fields generated by two patches ${E(0.1\epsilon)}\cup{E^-(0.1\epsilon)}$ and
$\Omega_1(t)\cup\Omega_1^{-}(t)$. Recall that $E(\epsilon)$ was introduced in (\ref{sglad}).

 We will
need to use the following lemma in which the picture above will be
rotated by $\pi/4$ degrees in the anticlockwise direction.  Let us
denote by $\Gamma_2(t),\Omega_2(t),$ and $E_2(0.1\epsilon)$ the resulting sets.
Clearly, the calculations done above indicate that the lower part of
$\Gamma_2(t)$ converges uniformly to the graph of the function $|x|$
however it is the precise form of this convergence that will play
the crucial role in comparing the velocity fields.

Fix $t$ and assume that some arc (let us call it $\Gamma_2$ as we
will later apply this lemma to the part of $\Gamma_2(t)$) lies above
the graph of $|x|$ and below the graph of a certain function $g(x)$
defined on $[-a,a]$ where, e.g., $a= 1/(e\sqrt 2)$.\vspace{1cm}

\begin{picture}(102,90)(75,-305)
        \allinethickness{0.254mm}\special{sh 0.3}\path(75,-215)(165,-215)(165,-305)(75,-305)(75,-215) 
\path(75,-305)(350,-305)

\path(75,-305)(75,-205)
        \allinethickness{0.254mm}\cbezier[500](95,-285)(85,-270)(85,-270)(85,-250) 
        \allinethickness{0.254mm}\cbezier[500](95,-285)(105,-295)(105,-295)(120,-295) 
        \allinethickness{0.254mm}\cbezier[500](85,-250)(90,-220)(90,-220)(145,-225) 
        \allinethickness{0.254mm}\cbezier[500](145,-225)(185,-230)(185,-230)(155,-275) 
        \allinethickness{0.254mm}\cbezier[500](155,-275)(140,-290)(140,-290)(120,-295) 
 \put(65,-315){\shortstack{$O$}}
 \put(85,-290){\shortstack{$\Gamma_1$}}
 \put(115,-260){\shortstack{$\Omega_1$}}
 \put(155,-300){\shortstack{$E$}}
\put(360,-305){\shortstack{\bf Figure 1}}
\end{picture}\bigskip\bigskip

\begin{lemma}
Rescale $g(x)$ as follows
\[
\widehat g(\widehat x)=\epsilon^{-1}g(\epsilon \widehat x), \quad |\widehat x|<a/\epsilon
\]
Assume that  $\widehat g$ satisfies the following properties:
\begin{itemize}
\item[(a)]{
 $0<C_1<\widehat g(\widehat x)<C_2, \, \widehat x\in [-1,1]$ with some absolute constants $C_{1(2)}$.}

\item[(b)]{ For $|\widehat{x}|>1$,

\begin{equation}\label{funkh}
|\widehat x|\leq \widehat g\leq |\widehat x|+h(\widehat x), \quad h(\widehat x)=
\overline{o}(|\widehat x|),\quad      \widehat x \to \infty
\end{equation}
}
\end{itemize}
Consider
\[
D(z,t)=\nabla^\perp \Delta^{-1} (\chi_{\Omega_2(t)\cup\Omega^-_2(t)}-\chi_{E_2(0.1\epsilon)\cup E^{-}_2(0.1\epsilon)})
\]
Then, we have the following estimate
\[
|D(z,t)|\lesssim |z|\left(1+\int\limits_{1<|\widehat x|<a/\epsilon}
\frac{h(\widehat x)+h(-\widehat x)}{|\widehat x|\cdot (|\widehat
x-\widehat z_1|+1)|}d\widehat x\right), \quad z=\epsilon \widehat z,
\, \widehat{z}=(\widehat{z}_1,\widehat{z}_2)
\]

\end{lemma}

\begin{proof}
Take $z\in \mathbb{C}^+$. From the central symmetry, we have
\[
\nabla^\perp \Delta^{-1} (\chi_{\Omega_2\cup\Omega_2^{-}}-\chi_{E_2(0.1\epsilon)\cup
E_2^-(0.1\epsilon)})(0)=0
\]
Now, let us use the following formula
\[
\left|\frac{\xi-z}{|\xi-z|^2}-\frac{\xi}{|\xi|^2}\right|=\frac{|z|}{|\xi-z|\cdot
|\xi|}
\]
to get
\begin{eqnarray*}
|D(z,t)|\lesssim |z|\left(1+\int\limits_{(E_2(0.1\epsilon)\backslash
\Omega_2)\cap \{|z|<a/10\}}\frac{d\xi}{|z-\xi|\cdot |\xi|}
+\hspace{3cm}\right.
\\
\left. +\int\limits_{(E_2(0.1\epsilon)\backslash \Omega_2)^-\cap
\{|z|<a/10\}}\frac{d\xi}{|z-\xi|\cdot |\xi|} \right)
\end{eqnarray*}
for any $ |z|<a/10$. Scale  by $\epsilon$ and notice
that $|\widehat{z}-\widehat{\xi}|\geq
|\widehat{z}_1-\widehat{\xi}_1|$ (if $\widehat z=(\widehat
z_1,\widehat z_2)$) so the first integral can be estimated by
\[
C\left(1+\int\limits_{1<|\widehat x|<a/\epsilon, |\widehat
x-\widehat z_1|>1}\frac{h(\widehat x)} {|\widehat x|\cdot (|\widehat
x-\widehat{z}_1|+1)}d\widehat x\right)
\]
and the integral over $|\widehat x-\widehat z_1|<1$ is uniformly
bounded due to (\ref{funkh}). For the second integral, we have a
similar bound with $h(-\widehat x)$ by symmetry.
\end{proof}
We immediately get
\begin{corollary}
If the function $h$ in (\ref{funkh}) satisfies
\begin{equation}
|h(\widehat x)|\lesssim |\widehat x|^{-\gamma}, \quad
\gamma>0\label{c1}
\end{equation}
for $|\widehat x|<\epsilon^{-\alpha}, \,\alpha\in (0,1)$ and
\begin{equation}
|\ln \epsilon| \max_{\epsilon^{-\alpha}<|\widehat x|<a/\epsilon}
\left|\frac{h(\widehat x)}{\widehat x}\right|<C\label{c2}
\end{equation}
then
\[
|D(z,t)|\lesssim |z|, \quad |z|\gtrsim \epsilon
\]
\end{corollary}
\begin{proof}
From (\ref{c1})
\[
\int_{1<|\widehat x|<\epsilon^{-\alpha}}\frac{h(\widehat x)}{|\widehat x|\cdot
(|\widehat x-\widehat z_1|+1)}d\widehat x\lesssim (|\widehat z_1|+1)^{-1}
\]
and (\ref{c2}) yields
\[
\int_{\epsilon^{-\alpha}<|\widehat x|<a/\epsilon}\frac{h(\widehat
x)}{|\widehat x|\cdot (|\widehat x-\widehat z_1|+1)}d\widehat x\lesssim |\ln
\epsilon| \max_{\epsilon^{-\alpha}<|\widehat x|<a/\epsilon}
\left|\frac{h(\widehat x)}{\widehat x}\right|<C
\]
\end{proof}
Now, let us apply this corollary to our situation. For this, we will use the results from
the previous section.\smallskip

{\bf Remark 1.} Notice that for each $\varrho$, the arc $\Gamma_1(t)$ satisfies the conditions
of the corollary with some $\alpha(\varrho)$. For $|\hat x|>\epsilon^{-\alpha}$, we can use
(\ref{ugolok}) and (\ref{predel2}) to get
\[
\left|\frac{h(\hat{x})}{\hat{x}}\right|\lesssim \Bigl(f(t/(1+2\delta))\Bigr)^\zeta,
\quad \zeta>0
\]
and
\[
\Bigl| \Bigl(\log f(t)\Bigr)
f(t/(1+2\delta))^\zeta\Bigr|\to
0
\]
Thus, we have
\begin{equation}\label{fact1}
\sup_{z} \left| \frac{D(z,t)}{|z|} \right|<C
\end{equation}

{\bf Remark 2. } If one repeats the estimates in lemma above, we have
\begin{equation}\label{nol}
|\nabla^\perp \Delta^{-1} (\chi_{E_2\cup E_2}-\chi_{E_2(0.1\epsilon)\cup E^{-}_2(0.1\epsilon)})|\lesssim |z|, \quad {\rm if }\quad|z|\gtrsim \epsilon
\end{equation}
\bigskip

\section{Construction of the vortex patch dynamics and proof of the main theorem.}

We first construct an  incompressible strain $\Psi$ which satisfies
the following properties  (see Figure 2 for the upper part of the
actual picture):

1. $\Psi$ is odd and is compactly supported.

2. Around the points $(4,4)$ and $(-2,2)$ it is the standard hyperbolic time-independent
flow (these are the domain $D_1$ and $D_6$). At $(-2,2)$, we choose
the separatrices to be $y_1=-x$ and $y_2=2+(x+2)$. The flow is attracting along $y_2$
and is repelling along $y_1$. To define the curve $\Gamma(t)$ in $D_1$ we need the following
result which will later guarantee necessary
initial conditions for the dynamics around $(0,0)$. Let the local
coordinates near $(-2,2)$ be denoted by $(\xi,\eta)$.

\begin{lemma}\label{ic}
Fix any $\varrho>0$ and consider the standard hyperbolic dynamics around the origin
\[
\left\{
\begin{array}{cc}
\dot \xi=\xi, & \xi(0)=\xi_0\\
\dot \eta=-\eta, & \eta(0)=\eta_0
\end{array}
\right.
\]
Let $G(\xi)=0$ for $\xi\leq 0$ and
$G(\xi)=\xi^{-1}\exp(-\xi^{-\varrho})$ for $\xi>0$. Consider the
evolution of the smooth curve $\Gamma(0)=\{(\xi,G(\xi)), |\xi|<1\}$
under this flow. Call it $\Gamma(t)=\{(\xi,G(\xi,t)), |\xi|<1\}$.
Then,
\begin{equation}
G(1,t)= e^{-e^{\varrho t}}\label{zapas2}
\end{equation}
\end{lemma}
\begin{proof}
This is a straightforward calculation. The point $(\xi_0,\eta_0)$
moves to $(\xi_0 e^t, \eta_0 e^{-t})$ in time $t$. Thus, the point
$(e^{-t}, G(e^{-t}))\in \Gamma(0)$ will move to $(1, e^{-e^{\varrho
t}})$ in time $t$.
\end{proof}
{\bf Remark 1.} Notice that the part of $\Gamma(t)$ that belongs to
the left half-plane does not change in time. Within the window
$|\xi|<1$, the curve $\Gamma(t)$ is always smooth and converges to the coordinate axis.\smallskip

3. Around the origin (domain $D_2$), we choose
\begin{equation}\label{rota}
\Psi(z,t)=\nabla^{\perp} \Lambda_s^{(t)}\left(\frac{y+x}{\sqrt
2},\frac{y-x}{\sqrt 2},t\right)
\end{equation}
where $\Lambda_s^{(t)}$ is $\Lambda_s$ modified in the $0.1 f(t)$--neighborhood of zero:
\[
\Lambda_s^{(t)}(z)=xy\log \Bigl(0.1f(t)Q_1\left(    \frac{z}{0.1 f(t)}  \right)\Bigr)
\]
where $Q_1(z)$ is smooth, positive, and $Q_1(z)=|z|\Phi(\phi)$ for $|z|>1$.
Clealry, $\Lambda_s^{(t)}=\Lambda_s$ for $|z|>0.1f(t)$ so the dynamics of the curve considered in the section \ref{cauchy} would be the same had we
studied the flow generated by the potential $\Lambda_s^{(t)}$ instead. One can also easily check that
\begin{equation}\label{chet}
\nabla^\perp \Lambda_s^{(t)}(z)=(-x,y)\log f(t) +O(|z|)
\end{equation}
for $|z|\lesssim f(t)$. We changed the coordinates in (\ref{rota})
as we want to rotate the picture described in section 3 by $\pi/4$
in the positive direction. We also modified the value of the
potential in the $0.1f(t)$ neighborhood of the origin to get rid
of the artificial singularity generated by the sharp corner in the
limiting configuration.\smallskip

4.  Between $D_1$ and $D_2$ the potential can be smoothly
interpolated. \smallskip

5. In $D_{3(5)}$, the flow is laminar with direction perpendicular
to the black segments and in the north-eastern direction. \smallskip

6. The potential between zones $D_2$ and $D_3$ can be
smoothly interpolated, as well as the potential between $D_5$ and
$D_6$. In the zone $D_7$, the potential is zero so the curve is
frozen. This zone again is interpolated smoothly between $D_1$ and
$D_6$.\smallskip

7. In the zone $D_4$, we construct non-stationary potential in the
following way (only in this zone the flow is essentially
time-dependent!). We need an  argument that allows an interpolation
between two laminar flows and guarantees the prescribed evolution of
the curve $\Gamma(t)$ in these laminar zones.  What we want is to
define dynamics in the regions $D_3,D_4,D_5$ right after the points
on the curve leave $D_2$. We need to define this dynamics in such a
way that the motion of $\Gamma(t)$ is localized to these regions
and, moreover, that it does not move in $D_5$. Once again, in $D_3$
and $D_5$ we postulate the flow to be laminar and then we want to
define it in $D_4$. We will do that in the local coordinates.\vspace{1cm}

\setlength{\unitlength}{0.254mm}
\begin{picture}(555,418)(30,-430)
        \allinethickness{0.254mm}\special{sh 0.3}\put(230,-350){\ellipse{160}{160}} 
        \allinethickness{0.254mm}\special{sh 0.3}\put(97,-227){\ellipse{135}{135}} 
        \allinethickness{0.254mm}\put(230,-420){\vector(0,1){145}} 
        \allinethickness{0.254mm}\put(155,-355){\vector(1,0){150}} 
        \allinethickness{0.254mm}\path(230,-355)(100,-230) 
        \allinethickness{0.254mm}\path(65,-265)(145,-190)\special{sh 1}\path(145,-190)(140,-193)(141,-194)(142,-195)(145,-190) 
        \allinethickness{0.254mm}\path(65,-195)(140,-270)\special{sh 1}\path(140,-270)(135,-267)(136,-266)(137,-265)(140,-270) 
        \allinethickness{0.254mm}\path(230,-355)(280,-310) 
        \allinethickness{0.254mm}\special{sh 0.3}\path(295,-160)(430,-160)(430,-295)(295,-295)(295,-160) 
        \allinethickness{0.254mm}\special{sh 0.3}\put(480,-115){\ellipse{100}{100}} 
        \allinethickness{0.254mm}\path(445,-145)(510,-80)\special{sh 1}\path(510,-80)(505,-83)(506,-84)(507,-85)(510,-80) 
        \allinethickness{0.254mm}\path(515,-145)(445,-80)\special{sh 1}\path(445,-80)(448,-85)(449,-84)(450,-83)(445,-80) 
        \allinethickness{0.254mm}\path(300,-235)(360,-290) 
        \allinethickness{0.254mm}\path(365,-165)(425,-220)

 \put(50,-233){\shortstack{$(-2,2)$}} %
 \put(90,-286){\shortstack{$D_1$}} 
       \put(215,-369){\shortstack{$O$}}

\put(190,-297){\shortstack{{\large
$\Gamma(t)$}}}

\put(230,-150){\shortstack{{\Huge
$\Omega(t)$}}}

        \put(245,-416){\shortstack{$D_2$}} 
        \put(335,-291){\shortstack{$D_3$}} 
        \put(410,-256){\shortstack{$D_4$}} 
        \put(410,-196){\shortstack{$D_5$}} 
        \put(505,-121){\shortstack{$D_6$}} 
        \put(15,-440){\shortstack{\bf Figure 2}}
 \put(435,-116){\shortstack{$(4,4)$}}
\put(240,-51){\shortstack{$D_7$}} 
        \allinethickness{0.1mm}\cbezier[1000](100,-230)(35,-165)(30,-165)(40,-125) 
        \allinethickness{0.1mm}\cbezier[1000](40,-125)(60,-50)(115,-45)(235,-20) 
        \allinethickness{0.1mm}\cbezier[1000](235,-20)(375,-5)(375,-10)(485,-35) 
        \allinethickness{0.1mm}\cbezier[1000](485,-35)(540,-50)(545,-50)(505,-85) 
        \allinethickness{0.1mm}\cbezier[1000](505,-85)(435,-155)(430,-160)(400,-190) 
        \allinethickness{0.1mm}\cbezier[1000](400,-190)(385,-205)(390,-200)(385,-205) 
        \allinethickness{0.1mm}\cbezier[1000](385,-205)(375,-215)(375,-220)(380,-230) 
        \allinethickness{0.1mm}\cbezier[1000](380,-230)(380,-240)(380,-240)(325,-245) 

        \allinethickness{0.1mm}\cbezier[1000](325,-245)(310,-250)(305,-255)(305,-255) 

        \allinethickness{0.1mm}\cbezier[1000](305,-255)(265,-295)(260,-295)(220,-305) 
        \allinethickness{0.1mm}\cbezier[1000](220,-305)(195,-305)(195,-305)(140,-265) 
        \allinethickness{0.1mm}\cbezier[1000](140,-265)(110,-240)(115,-240)(100,-230) 
\end{picture}\bigskip\bigskip

Assume that potential $\Lambda(z)=-y$ in $B=\{z: -1<x<0\}\cup \{z:
1<x<2\}\sim D_3\cup D_5$. This potential generates the laminar flow
\[
\dot\theta=\nabla \theta\cdot \nabla^\perp\Lambda
\]
where $\nabla^\perp\Lambda(z)=(-1,0)$. We want to define smooth
$\Lambda(z,t)$ in $D_4=\{z:0<x<1\}$ such that the resulting
$\Lambda(z,t)$ is smooth globally on $D_3\cup D_4\cup D_5$.
Moreover, given smooth decaying  $\nu(t)$ (e.g., $\nu\in
L^1(\mathbb{R}^+)$ is enough for decay condition), we need to define
a curve $\Gamma(0)=\{(x,\gamma(x,0))\}$ that evolves under this flow
$\Gamma(t)=\{(x,\gamma(x,t))\}$ such that $\gamma(0,t)=\nu(t)$ and
$\gamma(1,t)=0$. This function $\nu(t)$ is determined by $\Gamma(t)$
in the zone $D_2$ where it approaches the separatrix in the double
exponential rate. To be more precise, $\nu$ is proportional to the
distance from $\Gamma(t)$ to this separatrix in the area where $D_2$
and $D_3$ meet.

We will look for
\[
\Lambda(z,t)=-y-g_1(x)g_2(x-t)
\]
where $g_{1(2)}$ are smooth. Then, to guarantee the global
smoothness, we need $g_1(x)=0$ around $x=0$ and $x=1$. Now, take a
point $(0, \nu(T))$ and trace its trajectory for $t>T$. We have
\[
x(t,T)=t-T,\, t\in [T,T+1]
\]
and
\begin{eqnarray*}
y(t,T)=\nu(T)-\int_T^t
\Bigl(g_1'(\tau-T)g_2(\tau-T-\tau)+g_1(\tau-T)g_2'(\tau-T-\tau)\Bigl)d\tau,\\
t\in [T,T+1]
\end{eqnarray*}
Since we want $y(T+1,T)=0$ and $g_1$ to vanish on the boundary,
\[
\nu(T)=g_2'(-T)\int_T^{T+1} g_1(\tau-T)d\tau=g_2'(-T)\int_0^1
g_1(x)dx
\]
and this identity should hold for all $T>0$. Take any $g_1$ with
mean one, this defines $g_2$ on the negative half-line as long as
 we set $g_2(-\infty)=0$. We can continue it now to the whole line in a
smooth fashion to have $g_2$ globally defined. How do we define the
initial curve at $t=0$? We extend smooth $\nu(t)$ to $t\in [-1,0]$
arbitrarily and apply the procedure explained above to $t\in
[-1,\infty)$. The curve that we see at $t=0$ will be the needed
initial value for the dynamics that starts at $t=0$. It is only left
to mention that to localize the picture in the vertical direction we
can multiply $\Lambda(z,t)$ be a suitable cut--off in the $y$
direction. \bigskip

The part of the curve that is in $D_5$, $D_6$, $D_7$, and the
north-western part of $D_1$ is stationary, it does not move at all
(this is easy to ensure by  making this part of the curve the level
set of the stationary potential $\Lambda(z), \Psi=\nabla^\perp
\Lambda$). For the rest of the curve, it does change in time and the
flow is directed along it in the anti-clockwise direction.

Now that the explicit $\Psi(z,t)$ and the curve $\Gamma(t)=\partial\Omega(t)$ evolving under
this flow are defined, we are ready to prove theorem \ref{main}.
\begin{proof}{\it (Theorem \ref{main}).}
We have by construction
\[
\dot\theta=\nabla\theta\cdot \Psi(z,t)
\]
and $\theta(z,t)=\chi_{\Omega(t)}(z)+\chi_{\Omega^{-}(t)}(z)$. Let us define $S(z,t)$ by
\[
S(z,t)=\Psi(z,t)-\pi \nabla^\perp \Delta^{-1}\theta
\]
To show that this difference satisfies (\ref{lip}) we only need to consider the behavior of $\theta$ around $z=0$ since
 the contribution
from $(\Omega(t)\cup\Omega^-(t))\backslash B_{0.1 a}(0)$ to
$\nabla^\perp \Delta^{-1}\theta$ is of order $O(|z|)$ as immediately follows from the central symmetry.
We write 
\[
\Psi(z,t)-\pi \nabla^\perp \Delta^{-1}\theta=\Psi(z,t)-I+I-\pi \nabla^\perp \Delta^{-1}\theta
\]
where $I$ is obtained by replacing $\Omega(t)$ around the origin by $E_2(0.1f(t))$ configuration. Then, 
$
|I-\pi \nabla^\perp \Delta^{-1}\theta|\lesssim |z|
$
follows from (\ref{fact1}). In the $f(t)$--neighborhood of the origin, $\Psi(z,t)-I$ is at most $C|z|$ due to  
(\ref{sem}) and (\ref{chet}). For $|z|\gtrsim f(t)$, we can use (\ref{sbg}) and (\ref{nol}) to show that $|\Psi(z,t)-I|\lesssim |z|$. 
Thus, we have (\ref{lip}).  The uniform log-Lipschitz condition
 (\ref{log-lip}) immediately follows as well since the velocity generated by any patch does satisfy it.
The error $S(z,t)$ will in fact be much
smaller than $\Psi(z,t)$ around the origin and so can be considered as a small error or correction. It is odd as $\theta$ is even and is also
divergence free as the difference of two divergence free vector
fields. We get
\begin{equation}\label{e1}
\dot\theta=\nabla \theta\cdot \Bigl(\pi \nabla^\perp \Delta^{-1} \theta+S(z,t)\Bigr)
\end{equation}
and the theorem is proved as the dynamics of $\Omega(t)$ satisfies
\[
{\rm dist}(\Omega(t),\Omega^-(t))= 2{\rm
dist}(\Omega(t),0)\sim f(t)
\]
For $f(t)$ we have a double exponential decay due to (\ref{omg}) and (\ref{omg1}).
\end{proof}
{\bf Remark 2.} Strictly speaking, the Cauchy data for the evolution
in the domain $D_2$ will not be given by (\ref{zapas2}) as the flow
will distort it when moving between zones $D_1$ and $D_2$. However,
this leads only to a minor change (a fixed time increment, in fact),
so the same argument goes through.

{\bf Remark 3.} The parameter $\delta$ in the formulation of the
theorem \ref{main} corresponds to $\delta$ in (\ref{ohoh}) and can
be chosen arbitrarily from the interval $(0,1)$. The size of this
interval is determined by the parameters of the problem: the value
of vorticity and  initial size of the patches.\bigskip

\section{Appendix: approximate self-similar solution to the contour dynamics}

In this section, we address the following question: is it possible
to construct an approximate solution to the Euler dynamics of
patches such that the self-similarity will persist on the larger
set? We will do that in a rather artificial way as we already can
make the right guess about what the solution should be. The construction presented here  gives an independent (and even shorter) proof of the double-exponential merging but it does not explain the mechanism of the singularity formation and is less illuminating in our opinion.

\smallskip

Assume that the boundary of the simply-connected patch is
parameterized by $\gamma(s,t)$. Then, the velocity at every point of
the contour can be computed by (\cite{dc1}, formula (1))
\[
u(\gamma(\xi,t),t)=C\int_0^{2\pi} \log
|\gamma(\xi,t)-\gamma(s,t)|\gamma'_s(s,t)ds
\]
This is a simple corollary of the Gauss integration formula.

If we have a centrally symmetric pair of vortices interacting with each other and, like before,
 the part of $\Gamma(t)$ close to the origin can be parameterized by the function $y(x,t)$,
 then the equation for evolution reads

\begin{eqnarray}
\dot y(x,t)=C\int_{-0.5}^{0.5}
(y'(x,t)-y'(\xi,t))\log\left(\frac{(x-\xi)^2+(y(x,t)-y(\xi,t))^2}
{(x+\xi)^2+(y(x,t)+y(\xi,t))^2} \right)d\xi  \nonumber\\
 +r(x,y,t) \hspace{1cm}\label{for1}
\end{eqnarray}
where $r(x,y,t)$ is a contribution from those parts of $\Gamma(t)$
and $\Gamma^{-}(t)$ that are away from the origin. This $r(x,y,t)$
is therefore smooth and $r(0,0,t)=0$ by symmetry. Let us drop
$r(x,y,t)$ and try to find an approximate self-similar solution? In
other words, we want
 $y(x,t)=\epsilon(t)\phi(x/\epsilon(t))$ to satisfy (\ref{for1}) up to some smaller order
 correction. Substitution into (\ref{for1}) gives
\begin{equation}\label{back}
\frac{\dot\epsilon}{\epsilon}\left(\phi(\widehat{x})-\phi'(\widehat{x})
\widehat{x}\right)=\hspace{4cm}
\end{equation}
\[
C\int\limits_{-0.5\epsilon^{-1}}^{0.5\epsilon^{-1}}
(\phi'(\widehat{x})-\phi'(\widehat{\xi}))\log\left(\frac{(\widehat{x}-\widehat{\xi})^2+
(\phi(\widehat{x})-\phi(\widehat{\xi}))^2}
{(\widehat{x}+\widehat{\xi})^2+(\phi(\widehat{x})+\phi(\widehat{\xi}))^2}
\right)d\widehat{\xi}+E(\widehat x,t)
\]
where $E$ is an error we will control later on. Let us rewrite the
integral as follows (up to a constant multiple)
\[
\int_{-0.5\epsilon^{-1}}^{0.5\epsilon^{-1}}
(\phi'(\widehat{x})-\phi'(\widehat{\xi}))(\widehat x\widehat
\xi+\phi(\widehat x)\phi(\widehat \xi)) K(\widehat x,\widehat
\xi)d\widehat \xi
\]
with
\begin{eqnarray*}
K(\widehat x,\widehat \xi)=
\frac{1}{b-a}\int_a^b \frac{d\eta}{\eta},\hspace{4cm} \\
a=(\widehat x-\widehat \xi)^2+(\phi(\widehat x)-\phi(\widehat
\xi))^2, \quad b=(\widehat x+\widehat \xi)^2+(\phi(\widehat
x)+\phi(\widehat \xi))^2
\end{eqnarray*}
and so we get
\begin{equation}
k_1(\widehat x,t)\phi'(\widehat x)\widehat x +k_2(\widehat
x,t)\phi'(\widehat x)\phi (\widehat x)+k_3(\widehat x,t)\widehat
x+k_4(\widehat x,t)\phi(\widehat x)\label{e5}
\end{equation}
where the coefficients $k_j$ are defined correspondingly. Assume now
that $\phi$ satisfies the following assumptions:
\begin{itemize}
\item[(a)]{
 $\phi(\widehat x)$ is smooth}

\item[(b)]{ $0<C_1<\phi(\widehat x)<C_2$ for $\widehat x\in [-1,1]$}

\item[(c)]{
$\phi(\widehat x)=|\widehat x|+\rho(\widehat x)$ where $|\widehat
x|>1$ with
\[
|\rho(\widehat x)|<|\widehat x|^{-\gamma}, \quad |\rho'(\widehat
x)|<|\widehat x|^{-\gamma},\quad  \gamma>0 \]}
\end{itemize}
Let us estimate the coefficients $k_j$ now. We will handle $k_3$,
the analysis for $k_2$ is similar.
\begin{equation}\label{aniz}
k_3(\widehat x)=-\frac 14\int_{-0.5\epsilon^{-1}}^{0.5\epsilon^{-1}}
\frac{\xi \phi'( \xi)}{\xi \widehat x+\phi(\xi) \phi(\widehat x)}
\left(\int_a^b \frac{d\eta}{\eta}\right)d\xi
\end{equation}
Consider $\widehat x\in [1,0.5 \epsilon^{-1}]$, the other values can
be treated similarly. For the integral over the positive $\xi$ we
have (after the change of variables $\xi=\widehat x\xi_1$, recall
that $x=\epsilon \widehat x$)
\[
I_1=-\int\limits^{0.5x^{-1}}_{0} \frac{\xi_1 (1+\rho'(\xi_1 \widehat
x)) }{\xi_1+(1+\rho(\widehat
x)/\widehat{x})(\xi_1+\rho(\widehat{x}\xi_1)/\widehat{x}) }\cdot A\,
d\xi_1
\]
\[
A=\frac 14 \int_{(\xi_1-1)^2+(1-\xi_1+\rho(\widehat x)/{\widehat
x}-\rho(\widehat x\xi_1)/\widehat
x)^2}^{(1+\xi_1)^2+(1+\xi_1+\rho(\widehat x)/\widehat x+
\rho(\widehat x \xi_1 )/\widehat x)^2}\frac{d\eta}{\eta}
\]
For $A$, we have a representation
\[
A=\frac{1}{\xi_1}+\frac{\rho(\widehat x)}{2\widehat
x\xi_1}+O(\xi_1^{-2}), \quad \xi_1> 1
\]
so
\[
I_1=-\frac 12 \int\limits_{1}^{0.5 x^{-1}} \frac{d\xi_1}{\xi_1}
\Bigl(1+        O\left( \frac{\rho(\widehat x\xi_1)}{\widehat
x\xi_1} \right) \Bigr)\Bigl(1+\rho'(\xi_1 \widehat x)\Bigr)+\ldots
\]
\[
=0.5\log x+{O}(1)
\]
For the other integral, changing the sign in integration
\[
I_2=-\int\limits^{0.5\epsilon^{-1}}_{0} \xi (1+\rho'(-\xi
))\frac{1}{b-a}\int_a^b \frac{d\eta}{\eta} d\xi
\]
Here, we have
\begin{eqnarray*}
b=(\widehat x-\xi)^2+(\widehat x+\xi+\rho(\widehat
x)+\rho(-\xi))^2,\,\, a=(\widehat x+\xi)^2+(\widehat
x-\xi+\rho(\widehat x)-\rho(-\xi))^2
\end{eqnarray*}
As both $\widehat x$ and $\xi$ are large in the interesting regime,
we are in the situation when
\[
a, b>(\widehat x^2+\xi^2)/2
\]
so we can use the mean-value formula
\[
\frac{1}{b-a}\int_a^b
\frac{d\eta}{\eta}=\frac{1}{b}+\frac{b-a}{2\eta_1^2}, \quad
\eta_1\in (a,b)
\]
Substituting, we have two terms: $I_2=-(T_1+T_2)$.

\[
T_1=\int_0^{0.5x^{-1}} \xi_1(1+\rho'(-\widehat x\xi_1)) B^{-1}
d\xi_1
\]
where
\begin{eqnarray*}
B=2(1+\xi_1^2)+2(1+\xi_1)(\rho(\widehat x)/\widehat x+
\rho(-\widehat x\xi_1)/\widehat x)\\
+(\rho(\widehat x)/\widehat x+\rho(-\widehat x\xi_1)/\widehat x)^2
\end{eqnarray*}
Thus,
\[
T_1=\int_0^{0.5x^{-1}} \frac{\xi_1(1+\rho'(-\widehat
x\xi_1))}{2(1+\xi_1^2)}d\xi_1+{O}(1)=-0.5 \log x+{O}(1)
\]
For the other term, we have
\[
|T_2|\lesssim \int_0^{0.5\epsilon^{-1}} \xi \frac{\widehat
x|\rho(-\xi)|+\xi|\rho(\widehat x)|+ |\rho(\widehat
x)\rho(-\xi)|}{(\widehat x^2+\xi^2)^2}d\xi
\]
\[
\lesssim \int_1^{0.5\epsilon^{-1}} \frac{\xi d\xi}{(\widehat
x^2+\xi^2)^2} \left( \frac{\widehat
x}{\xi^\gamma}+\frac{\xi}{{\widehat x}^\gamma}+1\right)d\xi<C
\]
Combining all terms, we have
\[
k_3=\log x+{O}(1),\quad x>\epsilon
\]
For $x\sim 0$, we get $I_{1(2)}=0.5\log \displaystyle
\epsilon+O(1)$. These calculations show that
\begin{equation}
k_3=\left\{
\begin{array}{cc}
\log |x|+O(1),& \quad |x|>\epsilon\\
\log \epsilon+O(1),& \quad |x|<\epsilon
\end{array}
\right.
\end{equation}
Analogous estimates can be obtained for $k_2$. They yield
\begin{equation}
k_2= -\left\{
\begin{array}{cc}
\log |x|+O(1),& \quad |x|>\epsilon\\
\log \epsilon+O(1),& \quad |x|<\epsilon
\end{array}
\right.
\end{equation}
The estimates for other terms are
\[
|k_{1(4)}|=O(1)
\]
Indeed,
\[
k_1=\int_{-0.5\epsilon^{-1}}^{0.5\epsilon^{-1}}  \xi K(\widehat
x,\xi)d\xi, \quad k_4=-\int_{-0.5\epsilon^{-1}}^{0.5\epsilon^{-1}}
\phi'(\xi) \phi(\xi) K(\widehat x,\xi)d\xi
\]
and if one does the same analysis as we did for $k_3$ in
(\ref{aniz}), we will get the sum of two integrals: one over
positive $\xi$ and the other one over negative $\xi$. Each will have
the same large logarithmic leading term but they will come with
different signs now and so will cancel each other in the sum leaving
us with the  uniformly bounded error terms.

Thus (\ref{e5}) can be written as
\[
-\log\epsilon (\phi'\phi-\widehat x)+E
\]
where
\[
E=(k_2+\log\epsilon)\phi'\phi+(k_3-\log \epsilon)\widehat
x+k_1\widehat x\phi'(\widehat x)+k_4\phi(\widehat x)
\]
Going back to (\ref{back}) and choosing $C$ appropriately ($C<0$),
one wants to make the following choice for $\epsilon$ and $\phi$:
\[
\epsilon'=\epsilon \delta \log \epsilon,\quad
x-\phi'\phi=-\delta(\phi-x\phi'), \quad \delta\in (0,1)
\]
Take for $\epsilon(t)$ one particular solution
\[
\epsilon(t)=\exp(-e^{\delta t})
\] The equation for $\phi$ we had before (see (\ref{oldstory})) and so we have
\[
|\phi-x|^A|\phi+x|^B=1, \quad A=(1+\delta)/2,\quad
B=(1-\delta)/2
\]
and its solution $\phi$ trivially satisfies assumptions (a), (b),
(c) mentioned above.

For the original equation (\ref{for1}), the error one gets after
substituting $\epsilon\phi(x/\epsilon)$ amounts to $\epsilon E$
where
\[
\epsilon E\lesssim \epsilon\quad {\rm if}\quad |x|\lesssim \epsilon
\]
and
\[
|\epsilon E|\lesssim \epsilon+\epsilon |(\log
|x|-\log\epsilon)(\phi'\phi-\widehat x)|+ |x\phi'(\widehat x)
k_1|+|\epsilon\phi(\widehat x) k_4| \quad {\rm if} \quad |x|\gtrsim
\epsilon
\]
Therefore, for $|x|\gtrsim \epsilon$,
\[
|\epsilon E|\lesssim |z|+\epsilon
\left|\frac{x}{\epsilon}\right|^{-\gamma(\delta)} \log
\left|\frac{x}{\epsilon}\right|, \quad \gamma(\delta)>0
\]
and $z=(x,\epsilon\phi(\widehat x))$. Thus, we see that the error is
small again so it
 is possible to find the approximate solution with the self-similar scaling that holds
 on the ball of size $\sim 1$. We are not trying to
 make this picture global and define the incompressible strain on the whole $\mathbb{R}^2$
 which corresponds to the error $\epsilon E$ but we believe it is possible.

\section{Acknowledgment}
This research was supported by NSF grant DMS-1067413. The hospitality of the Institute
 for Advanced Study at
Princeton, NJ is gratefully acknowledged. The author thanks
A.~Kiselev, F.~Nazarov, and A.~Zlatos for the constant interest in
this work and A. Mancho for interesting comments on her preprint
\cite{mancho}.

\bigskip

\end{document}